\documentclass[10pt]{amsart} 
\usepackage{amsmath,amssymb,bm, color}

\usepackage{amsmath}
\usepackage{graphicx,float} 

\usepackage{mathrsfs}
\usepackage{bbm}
\usepackage{bm}
\usepackage{amsfonts,amssymb}
\usepackage{multirow}
\usepackage{lineno}
\usepackage{color}

\numberwithin{equation}{section}
 \newtheorem{theorem}{Theorem}[section]
 \newtheorem{lemma}[theorem]{Lemma}

\def\3bar{{|\hspace{-.02in}|\hspace{-.02in}|}}

\def\T{{\mathcal{T}}}

\def\cal#1{{\mathcal #1}}

\def\bw{{\mathbf{w}}}
\def\bx{{\mathbf{x}}}
\def\bf{{\mathbf{f}}}
\def\bu{{\mathbf{u}}}
\def\bv{{\mathbf{v}}}
\def\bn{{\mathbf{n}}}

\def\bvarphi{{\boldsymbol{\varphi}}}
\def\bbeta{{\boldsymbol{\beta}}}

\def\bsigma{{\boldsymbol{\sigma}}}

\newtheorem{algorithm}{Auto-Stabilized WG Algorithm}[section]

\numberwithin{equation}{section}

\def\3bar{{|\hspace{-.02in}|\hspace{-.02in}|}}
\def\p#1{\begin{pmatrix}#1\end{pmatrix}}
 \def\cal#1{\mathcal{#1}}
\def\ad#1{\begin{aligned}#1\end{aligned}}  \def\b#1{\mathbf{#1}} 
\def\a#1{\begin{align*}#1\end{align*}} \def\an#1{\begin{align}#1\end{align}} \def\t#1{\hbox{#1}}

\begin{document}

\title []
 {Auto-Stabilized Weak Galerkin Finite Element Methods for Biot's consolidation model on Non-Convex Polytopal Meshes}

  \author {Chunmei Wang}
  \address{Department of Mathematics, University of Florida, Gainesville, FL 32611, USA. }
  \email{chunmei.wang@ufl.edu}

\author {Shangyou Zhang}
\address{Department of Mathematical Sciences,  University of Delaware, Newark, DE 19716, USA}   \email{szhang@udel.edu}

\begin{abstract}
This paper presents an auto-stabilized weak Galerkin (WG) finite element method for the Biot's consolidation model within the classical displacement-pressure two-field formulation. Unlike traditional WG approaches, the proposed scheme achieves numerical stability without the requirement of  traditional  stabilizers. Spatial discretization is performed using weak Galerkin finite elements for both displacement and pressure approximations, while a backward Euler scheme is employed for temporal discretization to ensure a fully implicit and stable formulation. 
We establish the well-posedness of the resulting linear system at each time step and provide a rigorous error analysis, deriving optimal-order convergence. A significant merit of this WG scheme is its flexibility on general shape-regular polytopal meshes, including those with non-convex geometries. By utilizing bubble functions as a primary analytical tool, the method produces stable, oscillation-free pressure approximations without specialized treatment. Numerical experiments are presented to validate the theoretical convergence rates and demonstrate the computational efficiency and robustness of the auto-stabilized formulation.
\end{abstract}

 
%\linenumbers
\maketitle
 
\section{Introduction} 

Based on Biot’s theoretical framework, poroelasticity provides a coupled mathematical description of fluid flow within deformable porous media. This modeling is indispensable to various geophysical sectors, including   the detection of gas hydrates. While early investigations into this coupling can be traced back to the one-dimensional studies of Terzaghi \cite{1}, the generalized mathematical theory was established by Maurice Biot through several foundational publications \cite{2,3}. Biot’s model has since been extensively employed for the quantitative and qualitative analysis of poroelastic phenomena. In recent years, the numerical simulation and theoretical study of Biot’s equations have seen a significant resurgence, driven by diverse applications in biomechanics, medical science, petroleum engineering, and food processing.

In this study, we consider a porous medium that is linearly elastic, homogeneous, and isotropic, saturated by an incompressible Newtonian fluid. Under these assumptions, the quasi-static Biot model is formulated as a time-dependent system of partial differential equations (PDEs) for the solid displacement vector, \ $\b u$, \ and the fluid pressure, $p$, defined on a domain $\Omega\subset \mathbb R^d$, $1\leq d\leq 3$ with a regular boundary $\Gamma$:
\begin{eqnarray}\label{model1}
    -\nabla\cdot \bsigma +\nabla p&=& \bf, \\
    -\nabla\cdot \partial_t \bu+\nabla\cdot( K\nabla p)&=&g, \label{model2}
\end{eqnarray}
where $\bsigma=2\mu \epsilon (\bu)+\lambda \nabla\cdot \bu I$ and $\epsilon(\bu)=\frac{1}{2}(\nabla \bu+ \nabla \bu^T)$ represent the effective stress and strain tensors, respectively. The coefficients $\lambda$ and $\mu$ denote the Lamé parameters, $K$ is the hydraulic conductivity tensor, $\bf$ is the density of body forces, and the source term $g$ accounts for fluid extraction or injection. The partial derivative with respect to time for the displacement is denoted by $\partial_t \bu$.

To ensure a well-posed coupled system, specific boundary and initial conditions are required. A standard configuration of boundary conditions is defined as:
$$\bu=0, \qquad K\nabla p\cdot \bn=0, \qquad \text{on}\ \Gamma_c, $$
$$\bsigma \cdot \bn=\bbeta, \qquad  p =0, \qquad \text{on}\ \Gamma_t, $$
where $\bn$ is the unit outward normal vector. The boundary $\Gamma$ is partitioned into disjoint subsets $\Gamma_t$ and $\Gamma_c$ of non-null measure such that $\Gamma = \Gamma_t \cup \Gamma_c$. Furthermore, the following incompressibility constraint is imposed at the initial time $t = 0$:
\begin{equation}\label{ini}
   (\nabla \cdot\bu) (\bx, 0)=0, \qquad \bx\in \Omega.
\end{equation}

The mathematical foundations regarding the existence and uniqueness of solutions for Biot’s model were rigorously established by Showalter \cite{4} and Zenisek \cite{5}, while the well-posedness of nonlinear poroelastic variants has been addressed in works such as \cite{6}. Although analytical solutions exist for specific linear poroelastic problems \cite{7,8}, the inherent complexity of practical applications necessitates the use of robust numerical simulations. Among the prevalent discretization strategies, finite difference methods \cite{9, 10, 11} and finite element methods \cite{12, 13, 14, 15, 16} have been extensively documented.

A primary challenge in the numerical simulation of Biot's model is the emergence of nonphysical oscillations in the pressure field \cite{17, 18, 19, 20, 21}. These instabilities typically manifest in regimes characterized by low permeability or exceedingly small time steps at the onset of the consolidation process. Traditionally, such phenomena have been attributed to the violation of the inf-sup condition \cite{22}. Consequently, several studies have focused on inf-sup stable discretizations for displacement and pressure \cite{23, 24,25}. However, research has also indicated that the inf-sup condition alone may not guarantee numerical stability \cite{26,27}. Alternative explanations highlight the role of monotonicity; for instance, adding time-dependent stabilization terms has been shown to produce oscillation-free results using MINI \cite{28} or $P_1$-$P_1$ elements \cite{27}. Beyond standard two-field formulations, multi-field approaches (three-field and four-field) using continuous and discontinuous Galerkin methods have also been explored \cite{21,29, 30, 31, 32, 33}, with recent stabilization techniques proposed for three-field problems \cite{34} and mass-lumping nonconforming methods \cite{35} providing further stability.

The Weak Galerkin (WG) finite element method offers a sophisticated framework for solving  PDEs  by approximating differential operators through  a distribution based approach for piecewise polynomials. By utilizing specifically designed stabilizers, the WG method reduces regularity requirements, demonstrating significant versatility across various model PDEs \cite{wg1, wg2, wg3, wg4, wg5, wg6, wg7, wg8, wg9, wg10, wg11, wg12, wg13, wg14,   wg15, wg16, wg17, wg18, wg19, wg20, wg21, itera, wy3655}. A defining characteristic of the WG framework is its reliance on weak derivatives and weak continuity, granting it the flexibility to maintain stability and accuracy across diverse PDE systems, including previous applications to Biot's model \cite{muhu}.

This paper introduces a simplified, auto-stabilized WG formulation that supports both convex and non-convex polytopal elements. This approach has recently demonstrated success in Poisson \cite{autosecon}, biharmonic \cite{autobihar}, and linear elasticity \cite{autoelas} problems,   Maxwell equations \cite{automaxwell} and  Stokes problem \cite{autostokes}. A key innovation of the current work is the elimination of traditional stabilizers by employing higher-degree polynomials for the computation of discrete weak differential operators. This methodology preserves the global sparsity and size of the stiffness matrix while reducing implementation complexity. By integrating bubble functions as a primary analytical tool, the method extends the utility of WG to non-convex geometries.

We develop this WG method specifically for the Biot system \eqref{model1}-\eqref{model2}, incorporating a locking-free discretization for displacement \cite{wg18}. Combined with appropriately selected WG spaces for pressure, we theoretically demonstrate that the formulation is both locking-free and inf- sup stable using standard polynomials. The well-posedness of the system is established through a backward Euler temporal discretization. We prove optimal-order convergence under standard regularity assumptions and numerically demonstrate that the scheme produces oscillation-free pressure approximations without requiring mass lumping or manually tuned stabilization parameters.

The remainder of this paper is organized as follows. Section 2 defines the weak differential operators. Section 3 introduces the proposed WG formulation for Biot's model. Section 4 establishes the well-posedness of the discrete system at each time step. Section 5, we derive the error equations. 
Section 6 provides a rigorous error analysis and proves optimal-order convergence. Finally, Section 7 presents numerical experiments validating the stability and efficiency of the stabilizer-free scheme.

Throughout this paper, standard Sobolev notation is employed. Let $D$ denote an open, bounded domain in $\mathbb{R}^d$ with a Lipschitz continuous boundary. For any integer $s \geq 0$, the inner product, semi-norm, and norm in $H^s(D)$ are denoted by $(\cdot,\cdot)_{s,D}$, $|\cdot|_{s,D}$, and $\|\cdot\|_{s,D}$, respectively. When $D = \Omega$, the subscript is omitted. For $s=0$, the notations simplify to $(\cdot,\cdot)_D$, $|\cdot|_D$, and $\|\cdot\|_D$.

\section{Weak Operators and Discrete Formulations}

This section reviews the definitions of the weak gradient, weak divergence and weak strain tensor operators, along with their corresponding discrete formulations, as introduced in \cite{wg18, muhu}.

\subsection{Weak Function Spaces}
Consider a polytopal element $T$ with boundary $\partial T$. A weak function on $T$ is defined as a pair $\bv = \{\bv_0, \bv_b\}$, where $\bv_0 \in [L^2(T)]^d$ represents the interior values and $\bv_b \in [L^2(\partial T)]^d$ represents the boundary values. It is important to note that $\bv_b$ is treated as a variable independent of the trace of $\bv_0$ on $\partial T$.

The space of all weak functions on $T$, denoted by $V(T)$, is defined as follows:
\begin{equation*} 
V(T) = \left\{ \bv = \{\bv_0, \bv_b\} : \bv_0 \in [L^2(T)]^d, \bv_b \in [L^{2}(\partial T)]^d \right\}.
\end{equation*}

Similarly, the scalar weak function space $W(T)$ is defined as:
\begin{equation*} 
W(T) = \left\{ q = \{q_0, q_b\} : q_0 \in L^2(T), q_b \in L^{2}(\partial T) \right\}.
\end{equation*}

\subsection{Weak Differential Operators}
The weak gradient $\nabla_w \bv$ is a linear operator mapping $V(T)$ to the dual space of $[H^1(T)]^{d \times d}$. For any $\bv \in V(T)$, $\nabla_w \bv$ is defined as a bounded linear functional satisfying:
\begin{equation*} 
 (\nabla_w\bv, \bvarphi)_T = -(\bv_0, \nabla \cdot \bvarphi)_T + \langle \bv_b, \bvarphi \cdot \bn \rangle_{\partial T}, \quad \forall \bvarphi \in [H^1(T)]^{d \times d},
 \end{equation*}
where $\bn$ denotes the unit outward normal vector to $\partial T$.

The weak divergence $\nabla_w \cdot \bv$ is a linear operator mapping $V(T)$ to the dual space of $H^1(T)$. For $\bv \in V(T)$, the operator is defined such that:
\begin{equation*} 
 (\nabla_w \cdot \bv, w)_T = -(\bv_0, \nabla w)_T + \langle \bv_b \cdot \bn, w \rangle_{\partial T}, \quad \forall w \in H^1(T).
 \end{equation*} 

Furthermore, the weak gradient of a scalar $q \in W(T)$, denoted by $\nabla_w q$, is a linear operator mapping $W(T)$ to the dual space of $[H^1(T)]^d$, satisfying:
 \begin{equation*} 
 (\nabla_w q, \bw)_T = -(q_0, \nabla \cdot \bw)_T + \langle q_b, \bw \cdot \bn \rangle_{\partial T}, \quad \forall \bw \in [H^1(T)]^d.
 \end{equation*} 

\subsection{Discrete Weak Operators}
For any non-negative integer $r \ge 0$, let $P_r(T)$ denote the space of polynomials on $T$ with total degree at most $r$. The discrete weak gradient $\nabla_{w, r, T}\bv$ for $\bv \in V(T)$ is the unique polynomial in $[P_r(T)]^{d \times d}$ satisfying:
\begin{equation}\label{2.4}
(\nabla_{w, r, T}\bv, \bvarphi)_T = -(\bv_0, \nabla \cdot \bvarphi)_T + \langle \bv_b, \bvarphi \cdot \bn \rangle_{\partial T}, \quad \forall \bvarphi \in [P_r(T)]^{d \times d}.
 \end{equation}
If $\bv_0 \in [H^1(T)]^d$, applying integration by parts to the first term on the right-hand side of \eqref{2.4} yields:
 \begin{equation}\label{2.4new}
(\nabla_{w, r, T}\bv, \bvarphi)_T = (\nabla \bv_0, \bvarphi)_T + \langle \bv_b - \bv_0, \bvarphi \cdot \bn \rangle_{\partial T}, \quad \forall \bvarphi \in [P_r(T)]^{d \times d}.
 \end{equation} 

The discrete weak strain tensor is then defined as:
\begin{equation*}
\epsilon_{w, r, T}(\bv) = \frac{1}{2} (\nabla_{w, r, T}\bv + \nabla_{w, r, T} \bv^T).
\end{equation*}

The discrete weak gradient $\nabla_{w, r, T}q$ for $q \in W(T)$ is the unique polynomial in $[P_r(T)]^d$ such that:
 \begin{equation}\label{gradient}
(\nabla_{w, r, T}q, \bw)_T = -(q_0, \nabla \cdot \bw)_T + \langle q_b, \bw \cdot \bn \rangle_{\partial T}, \quad \forall \bw \in [P_r(T)]^{d}.
 \end{equation}
If $q_0 \in H^1(T)$, integration by parts leads to:
 \begin{equation}\label{gradientnew}
(\nabla_{w, r, T}q, \bw)_T = (\nabla q_0, \bw)_T + \langle q_b - q_0, \bw \cdot \bn \rangle_{\partial T}, \quad \forall \bw \in [P_r(T)]^{d}.
 \end{equation} 

Finally, the discrete weak divergence $\nabla_{w, r, T} \cdot \bv$ for $\bv \in V(T)$ is the unique polynomial in $P_r(T)$ satisfying:
 \begin{equation}\label{div}
(\nabla_{w, r, T} \cdot \bv, w)_T = -(\bv_0, \nabla w)_T + \langle \bv_b \cdot \bn, w \rangle_{\partial T}, \quad \forall w \in P_r(T).
 \end{equation}
For $\bv_0 \in [H^1(T)]^d$, this can be reformulated as:
 \begin{equation}\label{divnew}
(\nabla_{w, r, T} \cdot \bv, w)_T = (\nabla \cdot \bv_0, w)_T + \langle (\bv_b - \bv_0) \cdot \bn, w \rangle_{\partial T}, \quad \forall w \in P_r(T).
 \end{equation}

 \section{Auto-Stabilized Weak Galerkin Algorithms} \label{Section:WGFEM}

Let ${\cal T}_h$ denote a finite element partition of the domain $\Omega \subset \mathbb{R}^d$ into polytopal elements, which satisfies the shape regularity conditions established in \cite{wy3655}. Let ${\mathcal E}_h$ be the set of all edges or faces in ${\cal T}_h$, and let ${\mathcal E}_h^0 = {\mathcal E}_h \setminus \partial \Omega$ represent the subset of interior edges or faces. For each element $T \in {\cal T}_h$, let $h_T$ denote its diameter, and define the mesh size $h = \max_{T \in {\cal T}_h} h_T$.

We define the following constrained Sobolev spaces:
\begin{align*}
[H_c^1(\Omega)]^d &= \{ \bv \in [H^1(\Omega)]^d : \bv = 0 \text{ on } \Gamma_c \}, \\
H_t^1(\Omega) &= \{ q \in H^1(\Omega) : q = 0 \text{ on } \Gamma_t \}.
\end{align*}

The weak formulation of \eqref{model1}--\eqref{model2} is stated as follows: Find $\bu \in [H_c^1(\Omega)]^d$ and $p \in H_t^1(\Omega)$ such that
\begin{align*}
2\mu (\epsilon \bu, \epsilon \bv) + \lambda (\nabla \cdot \bu, \nabla \cdot \bv) - (\nabla \cdot \bv, p) &= (\bf, \bv) + (\bbeta, \bv)_{\Gamma_t},  \forall \bv \in [H_c^1(\Omega)]^d, \\
-(\nabla \cdot \partial_t \bu, q) - (K\nabla p, \nabla q) &= (g, q), \quad \forall q \in H_t^1(\Omega).
\end{align*}

\subsection{Weak Finite Element Spaces}
Let $k \geq 1$. To approximate the displacement $\bu$, we introduce the local weak finite element space:
\begin{equation*}
V(T) = \{ \bv = \{\bv_0, \bv_b\} : \bv_0 \in [P_k(T)]^d, \bv_b \in [P_{k}(e)]^d, e \subset \partial T \}.
\end{equation*}
The global weak finite element space $V_h$ is constructed by assembling the local spaces $V(T)$ for all $T \in {\cal T}_h$, while enforcing the continuity of the boundary component $\bv_b$ across the interior interfaces ${\mathcal E}_h^0$:
\begin{equation*}
V_h = \{ \{\bv_0, \bv_b\} : \{\bv_0, \bv_b\}|_T \in V(T), \forall T \in {\cal T}_h \}.
\end{equation*}
The subspace of $V_h$ incorporating the essential boundary conditions on $\Gamma_c$ is defined as:
\begin{equation*}
V_h^c = \{ \bv \in V_h : \bv_b|_{\Gamma_c} = 0 \}.
\end{equation*}

For the pressure field $p$, we define the local weak space:
\begin{equation*}
W(T) = \{ q = \{q_0, q_b\} : q_0 \in P_k(T), q_b \in P_k(\partial T) \}.
\end{equation*}
Similarly, the corresponding global space $W_h$ and its constrained subspace $W_h^t$ are given by:
\begin{align*}
W_h &= \{ q = \{q_0, q_b\} : \{q_0, q_b\}|_T \in W(T), \forall T \in {\cal T}_h \}, \\
W_h^t &= \{ q \in W_h : q_b|_{\Gamma_t} = 0 \}.
\end{align*}

\subsection{Discrete Operators and Bilinear Forms}
For simplicity, we let the discrete weak gradient $\nabla_w \bv$, the discrete weak divergence $\nabla_w \cdot \bv$, and the discrete weak strain tensor $\epsilon_w \bv$ denote the operators $\nabla_{w, r, T}\bv$, $\nabla_{w, r, T} \cdot \bv$, and $\epsilon_{w, r, T}\bv$, respectively, as defined in \eqref{2.4} and \eqref{div}. Specifically, for each $T \in {\cal T}_h$:
\begin{equation*}
(\nabla_w \bv)|_T = \nabla_{w, r, T}(\bv|_T), \quad (\nabla_w \cdot \bv)|_T = \nabla_{w, r, T} \cdot (\bv|_T), \quad (\epsilon_w \bv)|_T = \epsilon_{w, r, T} (\bv|_T).
\end{equation*}
Following the auto-stabilization strategy in \cite{autosecon}, the degree $r$ of the polynomial space is chosen as $r = k-1+2N$ for non-convex elements and $r = k-1+N$ for convex elements, where $N$ denotes the number of faces/edges of the element $T$. Details can be found in \cite{autosecon}.

We define the following bilinear forms:
\begin{align*}
a(\bu, \bv) &= 2\mu (\epsilon_w \bu, \epsilon_w \bv) + \lambda (\nabla_w \cdot \bu, \nabla_w \cdot \bv), \\
c(p, q) &= (K\nabla_w p, \nabla_w q), \\
b(\bv, q) &= (\nabla_w \cdot \bv, q).
\end{align*}

\subsection{Fully Discrete Scheme}
We propose a weak Galerkin method for the Biot's equations, employing the WG method for spatial discretization and the backward Euler scheme for temporal discretization. The full discretization is defined as follows: 
Find $\bu_h^n = \{\bu_0^n, \bu_b^n\} \in V_h^c$ and $p_h^n = \{p_0^n, p_b^n\} \in W_h^t$ such that
\begin{equation} \label{WG2}
\begin{cases}
a(\bu_h^n, \bv_h) - b(\bv_h, p_h^n) = (\bf(t_n), \bv_0) + \langle \bbeta(t_n), \bv_b \rangle_{\Gamma_t}, & \forall \bv_h \in V_h^c, \\
-b(\bar{\partial}_t \bu_h^n, q_h) - c(p_h^n, q_h) = (g(t_n), q_0), & \forall q_h \in W_h^t,
\end{cases}
\end{equation}
where $\bar{\partial}_t \bu_h^n = \frac{\bu_h^n - \bu_h^{n-1}}{\Delta t}$ and $\Delta t$ is the time step size.

At each time step $t_n$, the solution is obtained via the following scheme:
\begin{algorithm}\label{PDWG1}
Find $\bu_h = \{\bu_0, \bu_b\} \in V_h^c$ and $p_h \in W_h^t$ such that
\begin{equation}\label{WG}
\begin{cases}
a(\bu_h, \bv_h) - b(\bv_h, p_h) = (\bf(t_n), \bv_0) + \langle \bbeta, \bv_b \rangle_{\Gamma_t}, & \forall \bv_h \in V_h^c, \\
-b(\bu_h, q_h) - \Delta t c(p_h, q_h) = (\hat{g}, q_0), & \forall q_h \in W_h^t,
\end{cases}
\end{equation}
where the superscript $n$ is omitted for brevity and $\hat{g} = \Delta t g(t_n) - \nabla_w \cdot \bu_h^{n-1}$.
\end{algorithm}

\section{Well-posedness} 

Recall that ${\cal T}_h$ represents a shape-regular finite element partition of the domain $\Omega$. For any element $T \in {\cal T}_h$ and any function $\phi \in H^1(T)$, the following trace inequality holds \cite{wy3655}:
\begin{equation}\label{tracein}
 \|\phi\|^2_{\partial T} \leq C(h_T^{-1}\|\phi\|_T^2+h_T \|\nabla \phi\|_T^2).
\end{equation}
If $\phi$ is a polynomial on $T$, a simplified version of the trace inequality applies \cite{wy3655}:
\begin{equation}\label{trace}
\|\phi\|^2_{\partial T} \leq Ch_T^{-1}\|\phi\|_T^2.
\end{equation}

For any weak function $\bv = \{\bv_0, \bv_b\} \in V_h$, we define the energy norm:
\begin{equation}\label{vhnorm1}
\3bar \bv\3bar_{V_h}= \left(\sum_{T\in {\cal T}_h} 2\mu\|\epsilon_{w} \bv\|_T^2 + \lambda \|\nabla_w \cdot \bv\|_T^2\right)^{\frac{1}{2}},
\end{equation}
and the corresponding discrete $H^1$-semi-norm:
\begin{equation}\label{vhnorm2}
\3bar \bv\3bar_{1,h}= \left(\sum_{T\in {\cal T}_h} 2\mu\|\epsilon \bv_0\|_T^2 + \lambda \|\nabla \cdot \bv_0\|_T^2 + h_T^{-1}\|\bv_0 - \bv_b\|_{\partial T}^2\right)^{\frac{1}{2}}.
\end{equation}

Similarly, for any $q = \{q_0, q_b\} \in W_h$, we define the norm:
\begin{equation}\label{whnorm1}
\3bar q\3bar_{W_h}= \left(\sum_{T\in {\cal T}_h} K\|\nabla_{w} q\|_T^2\right)^{\frac{1}{2}},
\end{equation}
and the discrete $H^1$-semi-norm:
\begin{equation}\label{whnorm2}
\3bar q\3bar_{1,h}= \left(\sum_{T\in {\cal T}_h} K\|\nabla q_0\|_T^2 + h_T^{-1}  \|q_0 - q_b\|_{\partial T}^2\right)^{\frac{1}{2}}.
\end{equation}

% \begin{lemma}\label{norm1} \cite{autostokes}
% For $\bv = \{\bv_0, \bv_b\} \in V_h$, there exists a constant $C$ such that:
% $$
% \|\nabla \bv_0\|_T \leq C\|\nabla_w \bv\|_T.
% $$
% \end{lemma}

% By Lemma \ref{norm1} and the definition of the strain tensor $\epsilon \bv = \frac{1}{2}(\nabla \bv + \nabla \bv^T)$ and $\epsilon_w \bv = \frac{1}{2}(\nabla_w \bv + \nabla_w \bv^T)$, we obtain the following result:
% \begin{lemma}\label{norm11}\cite{autostokes}
% For $\bv = \{\bv_0, \bv_b\} \in V_h$, there exists a constant $C$ such that:
% $$
% \|\epsilon \bv_0\|_T \leq C\|\epsilon_w \bv\|_T.
% $$
% \end{lemma}

\begin{lemma}\cite{autostokes}
There exist positive constants $C_1$ and $C_2$ such that for any $\bv = \{\bv_0, \bv_b\} \in V_h$, the following norm equivalence holds:
\begin{equation}\label{normeq}
 C_1\|\bv\|_{1, h} \leq \3bar \bv\3bar_{V_h} \leq C_2\|\bv\|_{1, h}.
\end{equation}
\end{lemma}

% \begin{lemma}\label{norm13}\cite{autostokes}
% For $q = \{q_0, q_b\} \in W_h$, there exists a constant $C$ such that:
% $$
% \|\nabla q_0\|_T \leq C\|\nabla_w q\|_T.
% $$
% \end{lemma}

\begin{lemma}\cite{autostokes}
There exist positive constants $C_1$ and $C_2$ such that for any $q = \{q_0, q_b\} \in W_h$, the following norm equivalence holds:
\begin{equation}\label{normeq8}
 C_1\|q\|_{1, h} \leq \3bar q\3bar_{W_h} \leq C_2\|q\|_{1, h}.
\end{equation}
\end{lemma}

\begin{lemma}\label{cover}\cite{muhu}
For any $\bv, \bw \in V_h^c$, we have:
\begin{equation}\label{aform}
 a(\bu, \bv) \leq \3bar \bu\3bar_{V_h}\3bar \bv\3bar_{V_h}, \quad a(\bv, \bv) \leq \3bar \bv\3bar_{V_h}^2.
\end{equation}
For any $p, q \in W_h^t$, it holds that:
\begin{equation}
 c(p, q) \leq \3bar p\3bar_{W_h}\3bar q\3bar_{W_h}, \quad c(p, q) \leq \3bar q\3bar_{W_h}^2.
\end{equation}
\end{lemma}

\begin{lemma} \label{infsuplemma}\cite{wg18}
There exists a constant $\alpha > 0$, independent of $h$, such that for all $p \in W_h$:
\begin{equation}\label{infsup}
 \sup_{\bv \in V_h^c} \frac{(\nabla_w \cdot \bv, p)}{\3bar \bv\3bar_{V_h}} \geq \alpha \|p\|.
\end{equation}
\end{lemma}

\begin{lemma}
The semi-norms $\3bar \cdot\3bar_{V_h}$ and $\3bar \cdot\3bar_{W_h}$ define norms on $V_h^c$ and $W_h^t$, respectively.
\end{lemma}

\begin{proof}
Assume $\3bar \bv\3bar_{V_h} = 0$ for some $\bv \in V_h^c$. By the norm equivalence \eqref{normeq}, we have $\|\bv\|_{1, h} = 0$. This implies $\epsilon \bv_0 = 0$ and $\nabla \cdot \bv_0 = 0$ in each element $T$, with $\bv_0 = \bv_b$ on $\partial T$. Consequently, $\bv_0 \in RM(T)$ on each element, where $RM(T)$ is the space of rigid body motions defined by:
$$
 RM(T) = \{\textbf{a} + \eta \textbf{x} : \textbf{a} \in \mathbb{R}^d, \eta \in so(d)\},
$$
where $so(d)$ is the space of skew-symmetric $d \times d$ matrices. The condition $\bv_0|_e = \bv_b$ implies the continuity of $\bv_0$ across the entire domain $\Omega$. The boundary condition $\bv_b = 0$ on $\Gamma_c$ ensures $\bv_0 = 0$ on $\Gamma_c$. By the second Korn's inequality, it follows that $\bv_0 = 0$ in $\Omega$, which implies $\bv_b = 0$ and thus $\bv \equiv 0$.

Next, we show that $\3bar \cdot\3bar_{W_h}$ is a norm on $W_h^t$. Assume $\3bar q\3bar_{W_h} = 0$ for $q \in W_h^t$. By \eqref{normeq8}, we have $\|q\|_{1, h} = 0$, implying $\nabla q_0 = 0$ on each $T$ and $q_0 = q_b$ on $\partial T$. Thus, $q_0$ is constant on each element. The interface condition $q_0 = q_b$ ensures that $q_0$ is constant throughout $\Omega$. Since $q_b = 0$ on $\Gamma_t$, it follows that $q_0 \equiv 0$ and $q_b \equiv 0$, yielding $q \equiv 0$.
\end{proof}

Finally, the bilinear form $b(\cdot, \cdot)$ is bounded as follows:
\begin{equation}\label{bform}
 b(\bv, q) \leq C\3bar \bv\3bar_{V_h} \|q\|.
\end{equation}

By applying Lemmas \ref{cover}, \ref{infsuplemma}, and inequality \eqref{bform}, and invoking standard theory \cite{wg18}, we establish that the linear system \eqref{WG} is well-posed. We define the following comprehensive bilinear form for $(\bu, p), (\bv, q) \in V_h^c \times W_h^t$:
$$
 T(\bu, p; \bv, q) = a(\bu, \bv) - b(\bv, p) - b(\bu, q) - \Delta t c(p, q),
$$
associated with the weighted norm:
$$
\3bar (\bv, q)\3bar_{\Delta t}^2 = \3bar \bv\3bar_{V_h}^2 + \|q\|^2 + \Delta t \3bar q\3bar_{W_h}^2.
$$

\begin{lemma}
The bilinear form $T(\cdot, \cdot; \cdot, \cdot)$ satisfies the inf-sup condition:
\begin{equation}\label{infsup2}
 \sup_{(\bv, q) \in V_h^c \times W_h^t} \frac{T(\bu, p; \bv, q)}{\3bar (\bv, q)\3bar_{\Delta t}} \geq \zeta \3bar (\bu, p)\3bar_{\Delta t},
\end{equation}
where $\zeta > 0$ is independent of $h$ and $\Delta t$. Consequently, the discrete system \eqref{WG} is well-posed.
\end{lemma}

\begin{proof}
From the inf-sup condition \eqref{infsup}, for any $p = \{p_0, p_b\} \in W_h^t$, there exists $\bw \in V_h^c$ such that:
$$
 b(\bw, p) \geq \alpha \|p\|^2, \quad \text{with } \3bar \bw \3bar_{V_h} = \|p\|.
$$
For $\bu \in V_h^c$ and $p \in W_h^t$, we choose $\bv = \bu - \theta \bw$ and $q = -p$. Then:
\begin{equation*}
\begin{split}
 T(\bu, p; \bv, q) &= a(\bu, \bu - \theta\bw) - b(\bu - \theta\bw, p) + b(\bu, p) + \Delta t c(p, p) \\
 &= \3bar \bu\3bar_{V_h}^2 - \theta a(\bu, \bw) + \theta b(\bw, p) + \Delta t \3bar p\3bar_{W_h}^2 \\
 &\geq \3bar \bu\3bar_{V_h}^2 - \frac{1}{2}\3bar \bu\3bar_{V_h}^2 - \frac{\theta^2}{2} \3bar \bw\3bar^2_{V_h} + \theta \alpha \|p\|^2 + \Delta t \3bar p\3bar^2_{W_h} \\
 &= \frac{1}{2}\3bar \bu\3bar_{V_h}^2 + \left(\theta \alpha - \frac{\theta^2}{2}\right)\|p\|^2 + \Delta t \3bar p\3bar^2_{W_h}.
\end{split}
\end{equation*}
Setting $\theta = \alpha$, we obtain:
$$
 T(\bu, p; \bv, q) \geq \frac{1}{2}\3bar \bu\3bar_{V_h}^2 + \frac{\alpha^2}{2}\|p\|^2 + \Delta t \3bar p\3bar^2_{W_h} \geq \zeta_1 \3bar (\bu, p) \3bar_{\Delta t}^2,
$$
where $\zeta_1 = \min \{\frac{1}{2}, \frac{\alpha^2}{2}\}$. Furthermore, we observe that:
\begin{equation*}
 \begin{split}
  \3bar(\bv, q)\3bar_{\Delta t}^2 &= \3bar(\bu - \theta\bw, -p)\3bar_{\Delta t}^2 \\
  &\leq 2\3bar\bu\3bar^2_{V_h} + 2\theta^2\3bar \bw \3bar_{V_h}^2 + \|p\|^2 + \Delta t \3bar p\3bar_{W_h}^2 \\
  &\leq 2\3bar\bu\3bar^2_{V_h} + (2\alpha^2 + 1) \|p\|^2 + \Delta t \3bar p\3bar^2_{W_h} \leq \zeta_2\3bar(\bu, p)\3bar_{\Delta t}^2,
 \end{split}
\end{equation*}
where $\zeta_2 = \max \{2, 2\alpha^2 + 1\}$. Thus, \eqref{infsup2} holds with $\zeta = \zeta_1 \zeta_2^{-1/2}$. The continuity of $T(\cdot, \cdot; \cdot, \cdot)$ is readily verified. 

Furthermore, it is easy to show that the blinear form $B(\bu, p;\bv, q)$ is continuous, i.e.,
$$
B(\bu, p;\bv, q) \leq \3bar(\bu, p)\3bar_{\Delta t} \3bar(\bv, q)\3bar_{\Delta t}.$$

By Babuska theory, the linear system \eqref{WG} is well-posed.
\end{proof}

\section{Error Equations}

In this section, we derive the error estimates for the fully discrete scheme \eqref{WG2}. We assume that the initial data $\mathbf{u}_h^0$ satisfies the discrete divergence-free condition $\nabla_w \cdot \mathbf{u}_h^0 = 0$ as described in \eqref{ini}; however, alternative initial conditions may be considered without altering the fundamental analysis.

We begin by establishing the error equations for our model. Let $\mathcal{Q}_h$ denote the $L^2$ projection onto the space $P_r(T)$, where the degree $r$ is defined as:
\begin{itemize}
    \item $r = k-1+2N$ for non-convex elements $T$,
    \item $r = k-1+N$ for convex elements $T$,
\end{itemize}
where $N$ represents the number of faces (or edges) of the element $T$.

For each element $T \in \T_h$, let $Q_0$ denote the $L^2$ projection onto $P_k(T)$, and for each edge or face $e \subset \partial T$, let $Q_b$ denote the $L^2$ projection onto $P_{k}(e)$. For any $\bv \in [H^1(\Omega)]^d$, the projection $Q_h \bv$ is defined by:
$$
 (Q_h \bv)|_T := \{Q_0(\bv|_T), Q_b(\bv|_{\partial T})\}, \quad \forall T \in \T_h.
$$
An analogous definition holds for any $q \in H^1(\Omega)$:
$$
 (Q_h q)|_T := \{Q_0(q|_T), Q_b(q|_{\partial T})\}, \quad \forall T \in \T_h.
$$

The following lemma summarizes the key projection properties required for the subsequent analysis.
\begin{lemma}\label{Lemma5.1new}\cite{wg18}
For $\mathbf{u} \in [H^1(T)]^d$ and $p \in H^1(T)$, the following commutative properties hold:
\begin{align}
    \nabla_{w} Q_h \mathbf{u} &= \mathcal{Q}_h(\nabla \mathbf{u}), \label{pronew1} \\
    \nabla_{w} \cdot Q_h \mathbf{u} &= \mathcal{Q}_h(\nabla \cdot \mathbf{u}), \label{pronew2} \\
    \epsilon_{w}(Q_h \mathbf{u}) &= \mathcal{Q}_h(\epsilon(\mathbf{u})), \label{pronew3} \\
    \nabla_{w}(Q_h p) &= \mathcal{Q}_h(\nabla p). \label{pronew4}
\end{align}
\end{lemma}

\begin{lemma}\label{errorequa2} 
Let $(\mathbf{u}, p)$ be the exact solutions to the Biot's consolidation model equations \eqref{model1}--\eqref{model2}. For any $\mathbf{v}_h \in V_h^c$ and $q_h \in W_h^t$, the following error equations hold:
\begin{align} 
    a(Q_h \mathbf{u}, \mathbf{v}_h) - b(\mathbf{v}_h, \mathcal{Q}_h p) &= (\mathbf{f}, \mathbf{v}_0) + \langle \boldsymbol{\beta}, \mathbf{v}_b \rangle_{\Gamma_t} + \ell_1(\mathbf{u}, \mathbf{v}_h) + \ell_2(\mathbf{v}_h, p), \label{erroreqn2_1} \\
    -b(\partial_t Q_h \mathbf{u}, q_h) - c(Q_h p, q_h) &= (g, q_0) - \ell_3(p, q_h), \label{erroreqn2_2}
\end{align}
where the linear functionals $\ell_1, \ell_2$, and $\ell_3$ represent the consistency errors defined as: 
$$
\ell_1 (\bu, \bv_h)=\sum_{T\in {\cal T}_h}  
 2\mu\langle \bv_b-\bv_0,  ({\cal Q}_h-I) \epsilon \bu \cdot\bn \rangle_{\partial T}\\
    +\lambda\langle \bv_b-\bv_0,  ({\cal Q}_h-I)\nabla \cdot \bu \cdot\bn \rangle_{\partial T} ,
$$ 
$$
\ell_2 (  \bv_h, p)=\sum_{T\in {\cal T}_h}- \langle ( {\cal Q}_h -I)p, (\bv_b-\bv_0)\cdot \bn \rangle_{\partial T},
$$
$$
\ell_3(p, q_h)=\sum_{T\in {\cal T}_h}- \langle q_b-q_0,  ({\cal Q}_h-I)K\nabla  p \cdot\bn \rangle_{\partial T}.
$$
\end{lemma}

\begin{proof}

Using \eqref{pronew3}, standard  integration by parts, and setting $\bvarphi= {\cal Q}_h \epsilon \bu$ in  \eqref{2.4new}, we obtain 
\begin{equation*} 
\begin{split}
&\sum_{T\in {\cal T}_h} 2\mu (\epsilon_w Q_h\bu, \epsilon_w\bv_h)_T\\
=&\sum_{T\in {\cal T}_h} 2\mu ({\cal Q}_h(\epsilon \bu), \epsilon_w\bv_h)\\
=&\sum_{T\in {\cal T}_h} 2\mu(
\epsilon \bv_0, {\cal Q}_h(\epsilon \bu))_T+2\mu\langle \bv_b-\bv_0,  {\cal Q}_h \epsilon \bu \cdot\bn \rangle_{\partial T}\\
=&\sum_{T\in {\cal T}_h} 2\mu(
\epsilon \bv_0,  \epsilon \bu)_T+2\mu\langle \bv_b-\bv_0,  {\cal Q}_h \epsilon \bu \cdot\bn \rangle_{\partial T}\\
=&\sum_{T\in {\cal T}_h} -2\mu(\bv_0, \nabla\cdot(\epsilon \bu))+2\mu\langle  \epsilon \bu\cdot\bn, \bv_0\rangle_{\partial T}+2\mu\langle \bv_b-\bv_0,  {\cal Q}_h \epsilon \bu \cdot\bn \rangle_{\partial T}\\
=&\sum_{T\in {\cal T}_h} - (\bv_0, \nabla\cdot(2\mu\epsilon \bu))+ 2\mu\langle \bv_b-\bv_0,  ({\cal Q}_h-I) \epsilon \bu \cdot\bn \rangle_{\partial T}+2\mu\langle \epsilon\bu\cdot\bn, \bv_b\rangle_{\Gamma_t},
\end{split}
\end{equation*}
where we used $\sum_{T\in {\cal T}_h} \langle \epsilon\bu\cdot\bn, \bv_b\rangle_{\partial T}=\langle \epsilon\bu\cdot\bn, \bv_b\rangle_{\partial \Omega}= \langle \epsilon\bu\cdot\bn, \bv_b\rangle_{\Gamma_t}$ since  $\bv_b=0$ on $\Gamma_c$.

Using \eqref{pronew2}, standard  integration by parts, and setting $\bvarphi= {\cal Q}_h \epsilon \bu$ in  \eqref{divnew}, we obtain 
\begin{equation*} 
\begin{split}
&\sum_{T\in {\cal T}_h} \lambda (\nabla_w\cdot Q_h\bu, \nabla_w\cdot\bv_h)_T\\
=&\sum_{T\in {\cal T}_h}  \lambda ({\cal Q}_h(\nabla\cdot\bu), \nabla_w \cdot\bv_h)\\
=&\sum_{T\in {\cal T}_h} \lambda(
\nabla\cdot \bv_0, {\cal Q}_h(\nabla \cdot\bu))_T+\lambda\langle \bv_b-\bv_0,  {\cal Q}_h \nabla\cdot \bu \cdot\bn \rangle_{\partial T}\\
=&\sum_{T\in {\cal T}_h} \lambda(
\nabla \cdot \bv_0,  \nabla \cdot \bu)_T+\lambda\langle \bv_b-\bv_0,  {\cal Q}_h \nabla \cdot \bu \cdot\bn \rangle_{\partial T}\\
=&\sum_{T\in {\cal T}_h} -\lambda(\bv_0, \nabla\cdot(\nabla \cdot \bu))+\lambda\langle  \nabla \cdot \bu\cdot\bn, \bv_0\rangle_{\partial T}+\lambda\langle \bv_b-\bv_0,  {\cal Q}_h \nabla \cdot\bu \cdot\bn \rangle_{\partial T}\\
=&\sum_{T\in {\cal T}_h} - (\bv_0, \nabla\cdot(\lambda\nabla \cdot \bu))+\lambda\langle \bv_b-\bv_0,  ({\cal Q}_h-I)\nabla \cdot \bu \cdot\bn \rangle_{\partial T}+\lambda\langle  \nabla \cdot\bu\cdot\bn, \bv_b\rangle_{\Gamma_t},
\end{split}
\end{equation*}
where we used $\sum_{T\in {\cal T}_h} \langle \nabla \cdot\bu\cdot\bn, \bv_b\rangle_{\partial T}=\langle  \nabla \cdot\bu\cdot\bn, \bv_b\rangle_{\partial \Omega}=\langle  \nabla \cdot\bu\cdot\bn, \bv_b\rangle_{\Gamma_t}$ since  $\bv_b=0$ on $\Gamma_c$.

Using standard integration by parts and setting  $w={\cal Q}_h  p$ in  \eqref{divnew}, we obtain 
 \begin{equation*}\label{termm}
     \begin{split}
  & 
  \sum_{T\in {\cal T}_h} (\nabla_w \cdot \bv_h, {\cal Q}_h p)_T \\=& \sum_{T\in {\cal T}_h} (\nabla \cdot \bv_0,  {\cal Q}_h p)_T + \langle {\cal Q}_h p, (\bv_b-\bv_0)\cdot \bn \rangle_{\partial T}\\
  =& \sum_{T\in {\cal T}_h} (\nabla \cdot \bv_0,    p)_T + \langle {\cal Q}_h p, (\bv_b-\bv_0)\cdot \bn \rangle_{\partial T}\\
  =& \sum_{T\in {\cal T}_h}-( \bv_0,    \nabla p)_T 
+\langle p, \bv_0\cdot\bn\rangle_{\partial T}+ \langle {\cal Q}_h p, (\bv_b-\bv_0)\cdot \bn \rangle_{\partial T}\\    =& \sum_{T\in {\cal T}_h}-( \bv_0, \nabla p)_T + \langle ( {\cal Q}_h -I)p, (\bv_b-\bv_0)\cdot \bn \rangle_{\partial T},
\end{split}
 \end{equation*}
where we used   $\sum_{T\in {\cal T}_h} \langle p, \bv_b\cdot\bn\rangle_{\partial T}=\langle p, \bv_b\cdot\bn\rangle_{\partial \Omega}=0$ since $\bv_b=0$ on $\Gamma_c$ and $p=0$ on $\Gamma_t$.

Adding the above three equations and using \eqref{model1},  we have
\begin{equation*}
  \begin{split}
   &a(Q_h\bu, \bv_h)-b(\bv_h, {\cal Q}_h p) \\
   =&  \sum_{T\in {\cal T}_h} - (\bv_0, \nabla\cdot(2\mu\epsilon \bu))+ 2\mu\langle \bv_b-\bv_0,  ({\cal Q}_h-I) \epsilon \bu \cdot\bn \rangle_{\partial T}+2\mu\langle \epsilon\bu\cdot\bn, \bv_b\rangle_{\Gamma_t}\\
   &- (\bv_0, \nabla\cdot(\lambda\nabla \cdot \bu))+\lambda\langle \bv_b-\bv_0,  ({\cal Q}_h-I)\nabla \cdot \bu \cdot\bn \rangle_{\partial T}+\lambda\langle  \nabla \cdot\bu\cdot\bn, \bv_b\rangle_{\Gamma_t}\\
    &+( \bv_0, \nabla p)_T - \langle ( {\cal Q}_h -I)p, (\bv_b-\bv_0)\cdot \bn \rangle_{\partial T}\\
    =&(\bf, \bv_0)+\langle \bbeta, \bv_b\rangle_{\Gamma_t}+2\mu\langle \bv_b-\bv_0,  ({\cal Q}_h-I) \epsilon \bu \cdot\bn \rangle_{\partial T}\\
    &+\lambda\langle \bv_b-\bv_0,  ({\cal Q}_h-I)\nabla \cdot \bu \cdot\bn \rangle_{\partial T} - \langle ( {\cal Q}_h -I)p, (\bv_b-\bv_0)\cdot \bn \rangle_{\partial T}.
  \end{split}  
\end{equation*} This completes the proof of  \eqref{erroreqn2_1}.

Using \eqref{pronew2}, we obtain 
 \begin{equation*} 
     \begin{split}
  &\sum_{T\in {\cal T}_h} (\nabla_w \cdot \partial_t Q_h\bu, q_0)_T = 
  \sum_{T\in {\cal T}_h} ({\cal Q}_h(\nabla  \cdot  \partial_t \bu),  q_0)_T  =    \sum_{T\in {\cal T}_h} ( \nabla  \cdot  \partial_t \bu,  q_0)_T.
\end{split}
 \end{equation*}

 Using \eqref{pronew4}, standard  integration by parts, and setting $\bw= K{\cal Q}_h(\nabla p)$ in  \eqref{gradientnew}, we obtain 
\begin{equation*} 
\begin{split}
&\sum_{T\in {\cal T}_h}   (K\nabla_w Q_h p, \nabla_wq_h)_T\\
=&\sum_{T\in {\cal T}_h}    (K{\cal Q}_h(\nabla p), \nabla_w q_h)\\
=&\sum_{T\in {\cal T}_h} (
\nabla q_0, K{\cal Q}_h(\nabla p))_T+ \langle q_b-q_0, K {\cal Q}_h \nabla p \cdot\bn \rangle_{\partial T}\\
=&\sum_{T\in {\cal T}_h} (
\nabla q_0,  K\nabla p)_T+ \langle q_b-q_0,  K{\cal Q}_h \nabla p \cdot\bn \rangle_{\partial T}\\
=&\sum_{T\in {\cal T}_h} - (q_0, \nabla\cdot(K \nabla p))+ \langle  K\nabla  p\cdot\bn, q_0\rangle_{\partial T}+ \langle q_b-q_0,  K{\cal Q}_h \nabla  p \cdot\bn \rangle_{\partial T}\\
=&\sum_{T\in {\cal T}_h}  - (q_0, \nabla\cdot(K \nabla p))+ \langle q_b-q_0,  ({\cal Q}_h-I)K\nabla  p \cdot\bn \rangle_{\partial T},
\end{split}
\end{equation*}
where we used $\sum_{T\in {\cal T}_h} \langle K \nabla p\cdot\bn, q_b\rangle_{\partial T}=\langle  K\nabla p\cdot\bn, q_b\rangle_{\partial \Omega}=\langle  K \nabla p\cdot\bn, q_b\rangle_{\Gamma_c}=0$ since  $q_b=0$ on $\Gamma_t$ and $K\nabla p\cdot\bn=0$ on $\Gamma_c$.

Adding the above two equations gives
\begin{equation*} 
\begin{split}
&\sum_{T\in {\cal T}_h}- (\nabla_w \cdot \partial_t Q_h\bu, q_0)_T-(K\nabla_w Q_h p, \nabla_wq_h)_T
\\ =&\sum_{T\in {\cal T}_h}  -( \nabla  \cdot  \partial_t \bu,  q_0)_T  + (q_0, \nabla\cdot(K \nabla p))- \langle q_b-q_0,  ({\cal Q}_h-I)K\nabla  p \cdot\bn \rangle_{\partial T}\\
=& (g, q_0)- \langle q_b-q_0,  ({\cal Q}_h-I)K\nabla  p \cdot\bn \rangle_{\partial T},
\end{split}
\end{equation*}
where we used \eqref{model2}.
This completes  the second equation of \eqref{erroreqn2_2}.

This concludes the proof.

\end{proof}

\section{Error Estimates} 

\begin{lemma}\cite{wg21}\label{lem}
Let $\mathcal{T}_h$ be a finite element partition of the domain $\Omega$ satisfying the shape-regularity assumptions specified in \cite{wy3655}. For any $0 \leq s \leq 1$, $1 \leq m \leq k$, and $1 \leq n \leq 2N+k-1$, the following estimates hold: 
\begin{eqnarray}
\label{error1}
 \sum_{T\in {\cal T}_h} h_T^{2s}\|({\cal Q}_h-I)p\|^2_{s,T}&\leq& C  h^{2(n+1)}\|p\|^2_{n+1},\\
\label{error2}
\sum_{T\in {\cal T}_h}h_T^{2s}\|\bu- Q _0\bu\|^2_{s,T}&\leq& C h^{2(m+1)}\|\bu\|^2_{m+1},\\
\label{error3}\sum_{T\in {\cal T}_h}h_T^{2s}\|\nabla\bu-{\cal Q}_h(\nabla\bu)\|^2_{s,T}&\leq& C h^{2n}\|\bu\|^2_{n+1}.
\end{eqnarray}
 \end{lemma}
 \begin{lemma} 
If   $\bu\in [H^{k+1}(\Omega)]^d$, then there exists a constant 
$C$
 such that 
\begin{equation}\label{erroresti1}
\3bar \bu-Q_h\bu \3bar_{V_h} \leq Ch^{k}\|\bu\|_{k+1}.
\end{equation}
\end{lemma}
\begin{proof}
Using the property \eqref{2.4new}, the trace inequalities \eqref{tracein} and \eqref{trace}, the Cauchy--Schwarz inequality, and the estimate \eqref{error2} for $m=k$ and $s=0, 1$, we derive that for any $\boldsymbol{\varphi} \in [P_r(T)]^{d \times d}$:
\begin{equation*}
\begin{split}
&|\sum_{T\in {\cal T}_h} (\epsilon_w (\bu-Q_h\bu), \bvarphi)_T|\\
 =&  | \sum_{T\in {\cal T}_h} (\epsilon(\bu-Q_0\bu), \bvarphi)_T
        -\langle Q_b\bu-Q_0\bu, \bvarphi  \cdot\bn\rangle_{\partial T}|\\
\leq & (\sum_{T\in {\cal T}_h} \|\epsilon (\bu-Q_0\bu)\|_T )^{\frac{1}{2}}
       (\sum_{T\in {\cal T}_h}\|\bvarphi\|_T^2)^{\frac{1}{2}}\\
 & +(\sum_{T\in {\cal T}_h} \| Q_b\bu-Q_0\bu\|_{\partial T} ^2)^{\frac{1}{2}} 
      (\sum_{T\in {\cal T}_h}\|\bvarphi \cdot\bn\|_{\partial T}^2)^{\frac{1}{2}}\\
 \leq & (\sum_{T\in {\cal T}_h} \|\epsilon (\bu-Q_0\bu)\|_T )^{\frac{1}{2}}(\sum_{T\in {\cal T}_h}\|\bvarphi\|_T^2)^{\frac{1}{2}}
\\&+(\sum_{T\in {\cal T}_h} h_T^{-1}\| \bu-Q_0\bu\|_{ T} ^2+h_T \| \bu-Q_0\bu\|_{1,T} ^2)^{\frac{1}{2}} (\sum_{T\in {\cal T}_h}h_T^{-1}\|\bvarphi \|_{T}^2)^{\frac{1}{2}}\\
\leq & Ch^k\|\bu\|_{k+1} (\sum_{T\in {\cal T}_h} \|\bvarphi \|_{T}^2)^{\frac{1}{2}}.
\end{split}
\end{equation*}
Letting $\bvarphi=\epsilon_w (\bu-Q_h\bu)$ yields 
$$
\sum_{T\in {\cal T}_h} 2\mu (\epsilon_w (\bu-Q_h\bu), \epsilon_w (\bu-Q_h\bu))_T\leq 
 Ch^{2k}\|\bu\|^2_{k+1}.$$  

 Similarly, we have 
 $$
\sum_{T\in {\cal T}_h} \lambda (\nabla_w (\bu-Q_h\bu), \nabla_w (\bu-Q_h\bu))_T\leq 
 Ch^{2k}\|\bu\|^2_{k+1}.$$
 
 Combining the above two equations  completes the proof.
 
\end{proof}

 \begin{lemma} 
If   $p\in H^{k+1}(\Omega)$, then there exists a constant 
$C$
 such that 
\begin{equation}\label{erroresti12}
\3bar p-Q_hp \3bar_{W_h} \leq Ch^{k}\|p\|_{k+1}.
\end{equation}
\end{lemma}
\begin{proof}
Using the property \eqref{gradientnew}, the trace inequalities \eqref{tracein} and \eqref{trace}, the Cauchy--Schwarz inequality, and the estimate \eqref{error2} for $m=k$ and $s =0, 1$, we derive that for any $\mathbf{w} \in [P_r(T)]^d$:
\begin{equation*}
\begin{split}
&|\sum_{T\in {\cal T}_h} (\nabla_w (p-Q_h p), \bw)_T|\\=&  | \sum_{T\in {\cal T}_h} (\nabla(p-Q_0p), \bw)_T-\langle Q_b p-Q_0p, \bw \cdot\bn\rangle_{\partial T}|\\
\leq & (\sum_{T\in {\cal T}_h} \|\nabla (p-Q_0p)\|_T )^{\frac{1}{2}}(\sum_{T\in {\cal T}_h}\|\bw\|_T^2)^{\frac{1}{2}}\\
 & +(\sum_{T\in {\cal T}_h} \| Q_bp-Q_0p\|_{\partial T} ^2)^{\frac{1}{2}} (\sum_{T\in {\cal T}_h}\|\bw \cdot\bn\|_{\partial T}^2)^{\frac{1}{2}}\\
 \leq & (\sum_{T\in {\cal T}_h} \|\nabla(p-Q_0p)\|_T )^{\frac{1}{2}}(\sum_{T\in {\cal T}_h}\|\bw\|_T^2)^{\frac{1}{2}}\\
&+(\sum_{T\in {\cal T}_h} h_T^{-1}\|p-Q_0p\|_{T} ^2+h_T \|p-Q_0p\|_{1, T} ^2)^{\frac{1}{2}} (\sum_{T\in {\cal T}_h}h_T^{-1}\|\bw \cdot\bn\|_{  T}^2)^{\frac{1}{2}}\\
\leq & Ch^{k}\|p\|_{k+1} (\sum_{T\in {\cal T}_h} \|\bw \|_{T}^2)^{\frac{1}{2}}.
\end{split}
\end{equation*}
Letting $\bw=\nabla_w (p-Q_hp)$ yields 
$$
\sum_{T\in {\cal T}_h} K(\nabla_w (p-Q_h p), \nabla_w (p-Q_h p))_T\leq 
 Ch^{2k}\|p\|^2_{k+1}.$$ 

 This completes the proof. 
\end{proof}
\begin{lemma}
For any $\bu\in [H^{k+1}(\Omega)]^d$, $q\in H^{k+1}(\Omega)$, $\bv_h\in V_h$ and $q_h \in W_h$, the following estimates hold:
\begin{equation}\label{es11}
   |\ell_1(\bu, \bv_h)| \leq  Ch^k \|\bu\|_{k+1}\3bar \bv_h \3bar_{V_h}, 
\end{equation}
\begin{equation}\label{es21}
   |\ell_2(\bv_h, p)| \leq  Ch^{k+1} \|p\|_{k+1}\3bar \bv_h \3bar_{V_h},
\end{equation}
\begin{equation}\label{es31}
   |\ell_3(p, q_h)| \leq  Ch^{k} \|p\|_{k+1}\3bar q_h \3bar_{W_h}.
\end{equation}

\end{lemma}

\begin{proof}
Recall that $\mathcal{Q}_h$ denotes the $L^2$ projection operator onto the finite element space of piecewise polynomials of degree at most $r$, where $r = 2N+k-1$ for non-convex elements and $r = N+k-1$ for convex elements in the finite element partition. 

Using the Cauchy--Schwarz inequality, the trace inequality \eqref{tracein}, the norm equivalence \eqref{normeq}, and setting $n=k$ in \eqref{error3}, we obtain:
 \begin{equation*} 
\begin{split}
|\ell_1(\bu, \bv_h)|\leq &(\sum_{T\in {\cal T}_h}  
  h_T^{-1}\|\bv_b-\bv_0\|_{\partial T} ^2)^{\frac{1}{2}} (\sum_{T\in {\cal T}_h}  
 h_T \|({\cal Q}_h-I) \epsilon \bu \cdot \bn\|_{\partial T} ^2)^{\frac{1}{2}} \\
 &+(\sum_{T\in {\cal T}_h}  
  h_T^{-1}\|\bv_b-\bv_0\|_{\partial T} ^2)^{\frac{1}{2}} (\sum_{T\in {\cal T}_h}  
 h_T \|({\cal Q}_h-I) \nabla  \cdot\bu \cdot \bn\|_{\partial T} ^2)^{\frac{1}{2}} \\
\leq & \|\bv_h\|_{1, h} (\sum_{T\in {\cal T}_h}  
 \|({\cal Q}_h-I) \epsilon \bu \cdot \bn\|_{T} ^2+ h_T^2 \|({\cal Q}_h-I) \epsilon \bu \cdot \bn\|_{1, T} ^2)^{\frac{1}{2}}\\ 
 &+\|\bv_h\|_{1, h} (\sum_{T\in {\cal T}_h}  
 \|({\cal Q}_h-I) \nabla  \cdot\bu \cdot \bn\|_{T} ^2+ h_T^2 \|({\cal Q}_h-I) \nabla  \cdot \bu \cdot \bn\|_{1, T} ^2)^{\frac{1}{2}}\\ 
 \leq & Ch^k \|\bu\|_{k+1}\3bar \bv_h \3bar_{V_h}.
\end{split}
\end{equation*}
This completes the proof of \eqref{es11}.

Using the Cauchy-Schwarz inequality, the trace inequality \eqref{tracein}, letting $n=k$ in  \eqref{error1}, the norm equivalence \eqref{normeq},   we have
 \begin{equation*} 
\begin{split}
|\ell_2(\bv_h, p)|\leq &(\sum_{T\in {\cal T}_h}  
  h_T^{-1}\|\bv_b-\bv_0\|_{\partial T} ^2)^{\frac{1}{2}} (\sum_{T\in {\cal T}_h}  
 h_T \|({\cal Q}_h-I)p\|_{\partial T} ^2)^{\frac{1}{2}} \\
\leq & \|\bv_h\|_{1, h} (\sum_{T\in {\cal T}_h}  
 \|({\cal Q}_h-I)p\|_{T} ^2+ h_T^2 \|({\cal Q}_h-I)p\|_{1, T} ^2)^{\frac{1}{2}}\\ 
 \leq & Ch^{k+1} \|p\|_{k+1}\3bar \bv_h \3bar_{V_h}.
\end{split}
\end{equation*}
This completes the proof of \eqref{es21}.

Using the Cauchy-Schwarz inequality, the trace inequality \eqref{tracein}, letting $n=k$ in  \eqref{error1}, the norm equivalence \eqref{normeq8},   we have
 \begin{equation*} 
\begin{split}
|\ell_3(p, q_h)|\leq &(\sum_{T\in {\cal T}_h}  
  h_T^{-1}\|q_b-q_0\|_{\partial T} ^2)^{\frac{1}{2}} (\sum_{T\in {\cal T}_h}  
 h_T \|({\cal Q}_h-I)K\nabla p \cdot \bn\|_{\partial T} ^2)^{\frac{1}{2}} \\
\leq & \|q_h\|_{1, h} (\sum_{T\in {\cal T}_h}  
 \|({\cal Q}_h-I)K\nabla p \cdot \bn\|_{T} ^2+ h_T^2 \|({\cal Q}_h-I)K\nabla p \cdot \bn\|_{1, T} ^2)^{\frac{1}{2}}\\ 
 \leq & Ch^{k} \|p\|_{k+1}\3bar q_h \3bar_{W_h}.
\end{split}
\end{equation*}
This completes the proof of \eqref{es31}.
\end{proof}

Based on \eqref{erroreqn2_1}-\eqref{erroreqn2_2} and following the standard error analysis for time dependent problems as presented by Thomee \cite{50}, we define the elliptic projections $\bar{\mathbf{u}}_h \in V_h^c$ and $\bar{p}_h \in W_h^t$   such that for any $\mathbf{v}_h \in V_h^c$ and $q_h \in W_h^t$:
\a{
    a(\bar{\bu}_h, \bv_h)-b(\bv_h, \bar{p}_h)& 
    =a(Q_h\bu, \bv_h)-b(\bv_h, Q_hp)-\ell_1 (\bu, \bv_h)-\ell_2(\bv_h, p),
 \\
    c(\bar{p}_h, q_h)& =c(Q_h p, q_h)-\ell_3(p, q_h).
    }

Then we can split the errors between the $L^2$ projections $Q_h\bu$, $Q_h p$ and $\bu_h^n$, $p_h^n$ as following
$$
Q_h\bu(t_n)-\bu_h^n=(Q_h\bu(t_n)-\bar{\bu}_h(t_n))-(\bu_h^n-\bar{\bu}_h(t_n))=e_{\bu, 1}^n-e_{\bu, 2}^n, 
$$
$$
Q_h p(t_n)-p_h^n=(Q_h p(t_n)-\bar{p}_h(t_n))-(p_h^n-\bar{p}_h(t_n))=e_{p, 1}^n-e_{p, 2}^n.
$$

Next, we provide the error estimates for $\mathbf{e}_{\mathbf{u}, 1}^n$ and $e_{p, 1}^n$.

\begin{lemma}\label{le1} 
Assume that the finite element partition $\mathcal{T}_h$ is shape-regular. If $\mathbf{u} \in [H^{k+1}(\Omega)]^d$ and $p \in H^k(\Omega)$, then we have the following estimates:
\begin{equation}\label{eee1}
   \3bar e_{p, 1}^n \3bar_{W_h} \leq Ch^k\|p\|_{k+1},
\end{equation}
\begin{equation}\label{eee2}
   \3bar e_{\bu, 1}^n \3bar_{V_h} \leq Ch^k(\|\bu\|_{k+1}+\|p\|_{k+1}).
\end{equation}
\end{lemma}
\begin{proof}
    For any $q_h\in W_h^t$, using \eqref{es31}, we have 
    $$
    c( e_{p, 1}^n, q_h)=c(Q_hp, q_h)-c(\bar{p}_h, q_h)=\ell_3(p, q_h)\leq Ch^k\|p\|_{k+1}\3bar q_h\3bar_{W_h}.
    $$
  This gives \eqref{eee1}.

    For any $\bv_h\in V_h^c$, using \eqref{bform}, \eqref{es11} and \eqref{es21} ,  we have
\begin{equation*}
    \begin{split}
        a(e_{\bu, 1}^n, \bv_h)=&a(Q_h\bu, \bv_h)-a(\bar{\bu}_h, \bv_h)\\
        =&b(\bv_h, e_{p, 1}^n)+\ell_1(\bu, \bv_h)+l_2(\bv_h, p)\\
        \leq & C\|e_{p, 1}^n\| \3bar\bv_h\3bar_{V_h}+Ch^k(\|\bu\|_{k+1}+\|p\|_{k+1})\3bar \bv_h\3bar_{V_h}.
    \end{split}
\end{equation*}
 This gives \eqref{eee2}.
\end{proof}

\begin{lemma} \label{le2} 
Assume that the finite element partition $\mathcal{T}_h$ is shape-regular. Then for $\mathbf{u} \in [H^{k+1}(\Omega)]^d$ and $p \in H^{k+1}(\Omega)$, the following estimates hold:
\begin{equation}\label{eee3}
   \3bar \partial_t e_{p, 1}^n \3bar_{W_h} \leq Ch^k\|\partial_t p\|_{k+1},
\end{equation}
\begin{equation}\label{eee4}
   \3bar \partial_t e_{\bu, 1}^n \3bar_{V_h} \leq Ch^k(\|\partial_t \bu\|_{k+1}+\|\partial_t p\|_{k+1}).
\end{equation}
\end{lemma}
\begin{proof}
    This Lemma can be proved similar to Lemma \ref{le1}. 
\end{proof}
Now we estimate the error $\mathbf{e}_{\mathbf{u}, 2}^n$ and $e_{p, 2}^n$. The overall error estimates can be derived by the triangle inequality. In order to simplify the notation, we introduce the following norm on the weak finite element spaces $V_h$ and $W_h$:
$$
\3bar (\bu, p)\3bar^2=\3bar \bu\3bar_{V_h}^2+\Delta t \3bar p\3bar_{W_h}^2.
$$

\begin{lemma}
    Let ${\cal R}_{\bu}^n =\partial _{t} Q_h \bu(t_n)-\frac{\bar{\bu}_h(t_n)-\bar{\bu}_h(t_{n-1})}{\Delta t}$, 
    we have
    \begin{equation} \label{errorestimate}
        \3bar (e_{\bu, 2}^n, e_{p, 2}^n)\3bar\leq C(\3bar e_{p, 2}^0\3bar_{V_h}+C\Delta t \sum_{j=1}^n \3bar {\cal R}_{\bu}^j\3bar_{V_h}).
    \end{equation}
\end{lemma}
\begin{proof}
    Using the definitions of $\bar{\bu}_h$ and $\bar{p}_h$ gives
    \begin{equation}\label{s1}
        a(e_{\bu, 2}^n, \bv_h)-b(\bv_h, e_{p, 2}^n)=0,
    \end{equation}
\begin{equation}\label{s2}
    -b(\bar{\partial} e_{\bu, 2}^n, q_h)-c(e_{p, 2}^n, q_h)=-b({\cal R}_{\bu}^n, q_h).
\end{equation}
Letting $\bv_h= \bar{\partial} e_{\bu, 2}^n$ in \eqref{s1} and $q_h= e_{p, 2}^n$ in \eqref{s2} and using \eqref{aform} and \eqref{bform}, we have
\begin{equation*}
    \begin{split}
        \3bar e_{\bu, 2}^n\3bar_{V_h}^2+\Delta t \3bar e_{p, 2}^n\3bar_{W_h}^2 &=a(e_{\bu, 2}^n, e_{\bu, 2}^{n-1})+\Delta t b({\cal R}_{\bu}^n, e_{p, 2}^{n})\\
        & \leq \3bar e_{\bu, 2}^n\3bar_{V_h}\3bar e_{\bu, 2}^{n-1}\3bar_{V_h}+C\Delta t \3bar {\cal R}_{\bu}^n\3bar_{V_h} \|e_{p, 2}^n\|. 
    \end{split}
\end{equation*}

Using the inf-sup condition \eqref{infsup} results in 
\begin{equation*}
    \|e_{p, 2}^n\|\leq Csup_{\bv_h\in V_h^c} \frac{b(\bv_h,  e_{p, 2}^n)}{\3bar \bv_h\3bar_{V_h}}=Csup_{\bv_h\in V_h^c} \frac{a( e_{\bu, 2}^n, \bv_h)}{\3bar \bv_h\3bar_{V_h}}=C\3bar e_{\bu, 2}^n \3bar_{V_h}.
\end{equation*}
Thus, 
\begin{equation}\label{aa1}
     \3bar e_{\bu, 2}^n\3bar_{V_h}^2+\Delta t \3bar e_{p, 2}^n\3bar_{W_h}^2\leq \3bar e_{\bu, 2}^n\3bar_{V_h}\3bar e_{\bu, 2}^{n-1}\3bar_{V_h}+C\Delta t \3bar {\cal R}_{\bu}^n\3bar_{V_h} \3bar e_{\bu, 2}^n \3bar_{V_h}.
\end{equation}
This implies
\begin{equation} \label{aa2}
 \3bar e_{\bu, 2}^n\3bar_{V_h} \leq  \3bar e_{\bu, 2}^0\3bar_{V_h}+C\Delta t \sum_{j=1}^n \3bar   {\cal R}_{\bu}^j\3bar_{V_h}.
\end{equation}
Combining \eqref{aa1} and \eqref{aa2} gives the error estimate \eqref{errorestimate}.

\end{proof}

Applying the same procedure of Lemma 8 in  \cite{27} leads to  
\begin{equation}\label{eee5}
    \begin{split}
        \sum_{j=1}^n \3bar {\cal R}_{\bu}^j\3bar_{V_h} \leq &C(\int_0^{t_n} \3bar \partial_{tt} Q_h\bu\3bar_{V_h}dt+\frac{1}{\Delta t} \int_0^{t_n} \3bar \partial_t  e_{\bu, 1}\3bar_{V_h}dt)\\
         \leq &C(\int_0^{t_n} \| \partial_{tt}  \bu\|_{1}dt+\frac{1}{\Delta t} \int_0^{t_n} \3bar \partial_t  e_{\bu, 1}\3bar_{V_h}dt).
    \end{split}
\end{equation}
Combining all the results above, we obtain the following theorem regarding the error estimates:

\begin{theorem}
  Assume that the finite element partition ${\cal T}_h$ is shape regular and $\bu(t)\in L^{\infty}((0, T], [H^{k+1}(\Omega)]^d)$, 
  $\partial_t \bu(t)\in L^{1}((0, T], [H^{k+1}(\Omega)]^d)$,  
  $\partial_{tt} \bu(t)\in L^{1}((0, T],$ $ [H^{k+1}(\Omega)]^d)$, 
   $p(t)\in L^{\infty}((0, T], H^{k+1}(\Omega))$,   $\partial_t p(t)\in L^{1}((0, T], H^{k+1}(\Omega))$, 
    the following error estimate holds  \begin{equation}\label{es1}
       \begin{split}
           &\3bar (Q_h \bu(t_n)-\bu_h^n, Q_h p(t_n)-p_h^n)\3bar\\
           \leq & C(\3bar  e_{\bu, 2}^0\3bar_{V_h}+\Delta t \int_0^{t_n} \| \partial_{tt}  \bu\|_{1}dt \\&+h(\|\bu\|_2+\|p\|_2+\int_0^{t_n} (\|\partial_t\bu\|_2+\|\partial_t p\|_2)dt)).
       \end{split}
   \end{equation}
As a result, the following error estimate holds
 \begin{equation}\label{es2}
       \begin{split}
           &\3bar (\bu(t_n)-\bu_h^n,  p(t_n)-p_h^n)\3bar\\
           \leq & C(\3bar  e_{\bu, 2}^0\3bar_{V_h}+\Delta t \int_0^{t_n} \| \partial_{tt}  \bu\|_{1}dt   \\&+h(\|\bu\|_2+\|p\|_2+\int_0^{t_n} (\|\partial_t\bu\|_2+\|\partial_t p\|_2)dt)).
       \end{split}
   \end{equation}
\end{theorem}
\begin{proof}
   From \eqref{eee1}, \eqref{eee2}, \eqref{errorestimate}, \eqref{eee5}, and triangular inequality, it is easy to derive  \eqref{es1}.      Using  
triangular inequality,  \eqref{es1}, and the error estimates of $L^2$ projections gives \eqref{es2}.

\end{proof}

\section{Numerical tests}

We solve first the Biot's model \eqref{model1}--\eqref{model2} on the unit square domain 
  $\Omega=(0,1)\times(0,1)$, where 
\an{\label{nu}  \mu&=\frac E{2(1+\nu)}, \quad \lambda=\frac{\nu E}{(1+\nu)(1-2\nu)}, \quad
   E=1, \quad  \nu=\nu_0, \quad K=1, \\ 
   \nonumber  f_1&= \pi^2[(3\mu+\lambda)\sin(\pi x)\sin(\pi y)
          -(\lambda+\mu) \cos(\pi x)\cos(\pi y)], \\
  \nonumber  \b f&= \t{e}^{-t}\p{f_1\\f_1}-\t{e}^{-t}\p{0\\ \pi \sin (\pi y)}, \quad \     \t{and}  \\
  \nonumber      g&= \t{e}^{-t}\pi [\cos(\pi x)\sin(\pi y)
          + \sin(\pi x)\cos(\pi y) -\pi K\cos(\pi y)]. }
The boundary conditions are 
\a{ \b u&=\b 0 \ \t{on} \ \partial\Omega; \\
       p&= 0 \ \t{on } \ \Gamma_t=\{ (x,1) : 0\le x \le 1\}; 
      \quad \ \partial_{\b n} p = 0 \ \t{on } \ \partial\Omega \setminus \Gamma_t.  }
The initial conditions are taken from the exact solution:
\an{\label{s-1} \ad{ \b u&= \t{e}^{-t}\sin(\pi x)\sin(\pi y) \p{1\\1},   \\
       p&= \t{e}^{-t}  [\cos(\pi  y)+1].  } }

The solution in \eqref{s-1} is approximated by the weak Galerkin finite element
   $P_k$-$P_k$/$P_{k+1}^2$ (for $\{u_0, u_b\}$/$\nabla_w$), $k= 1,2,3$, on nonconvex  
    grids shown in Figure \ref{f21}.
The errors and the computed orders of convergence are listed in Tables \ref{t1}--\ref{t3}.
The optimal order of convergence is achieved in every case.

\begin{figure}[H]
 \begin{center}\setlength\unitlength{1.0pt}
\begin{picture}(360,120)(0,0)
  \put(15,108){$G_1$:} \put(125,108){$G_2$:} \put(235,108){$G_3$:} 
  \put(0,-420){\includegraphics[width=380pt]{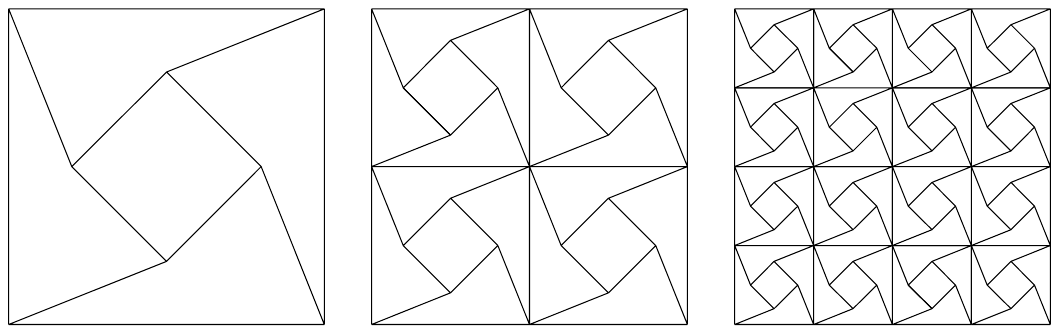}}  
 \end{picture}\end{center}
\caption{The nonconvex polygonal grids used in Tables \ref{t1}--\ref{t3}. }\label{f21}
\end{figure}

  \begin{table}[H]
  \caption{By the $P_1$ element for \eqref{s-1}  on Figure \ref{f21} grids.} \label{t1}
\begin{center}  
   \begin{tabular}{c|rr|rr|rr}  
 \hline 
$G_i$ &\multicolumn{2}{c|}{$ \|Q_h \b u - \b  u_h \| $   $O(h^r)$ }&
 \multicolumn{2}{c|}{$\| \nabla_w( Q_h \b u- \b u_h )\| $ $O(h^r)$ }
   & \multicolumn{2}{c}{$\|\nabla_w(Q_h p-p_h)\|$ $O(h^r)$} \\ \hline 
&\multicolumn{6}{c}{ By the $P_1$-$P_1$/$P_{2}^2$ weak Galerkin finite element, $\nu=0.25$ in \eqref{nu}.}
\\  \hline 
 3&    0.896E-02 &  2.0&    0.436E+00 &  1.3&    0.255E+00 &  1.1 \\
 4&    0.223E-02 &  2.0&    0.204E+00 &  1.1&    0.126E+00 &  1.0 \\
 5&    0.558E-03 &  2.0&    0.100E+00 &  1.0&    0.626E-01 &  1.0 \\
 \hline  
&\multicolumn{6}{c}{ By the $P_1$-$P_1$/$P_{2}^2$ weak Galerkin finite element, $\nu=0.499$ in \eqref{nu}.}
\\  \hline 
 3&    0.181E-01 &  4.5&    0.130E+01 &  2.6&    0.659E+00 &  1.1 \\
 4&    0.260E-02 &  2.8&    0.528E+00 &  1.3&    0.325E+00 &  1.0 \\
 5&    0.618E-03 &  2.1&    0.256E+00 &  1.0&    0.162E+00 &  1.0 \\
  \hline  
\end{tabular} \end{center}  \end{table}

  \begin{table}[H]
  \caption{By the $P_2$ element for \eqref{s-1}  on Figure \ref{f21} grids.} \label{t2}
\begin{center}  
   \begin{tabular}{c|rr|rr|rr}  
 \hline 
$G_i$ &\multicolumn{2}{c|}{$ \|Q_h \b u - \b  u_h \| $   $O(h^r)$ }&
 \multicolumn{2}{c|}{$\| \nabla_w( Q_h \b u- \b u_h )\| $ $O(h^r)$ }
   & \multicolumn{2}{c}{$\|\nabla_w(Q_h p-p_h)\|$ $O(h^r)$} \\ \hline 
&\multicolumn{6}{c}{ By the $P_2$-$P_2$/$P_{3}^2$ weak Galerkin finite element, $\nu=0.25$ in \eqref{nu}.}
\\  \hline 
 3&    0.128E-02 &  3.5&    0.142E+00 &  2.7&    0.530E-01 &  2.3 \\
 4&    0.134E-03 &  3.3&    0.269E-01 &  2.4&    0.124E-01 &  2.1 \\
 5&    0.158E-04 &  3.1&    0.605E-02 &  2.2&    0.305E-02 &  2.0 \\
 \hline  
&\multicolumn{6}{c}{ By the $P_2$-$P_2$/$P_{3}^2$ weak Galerkin finite element, $\nu=0.499$ in \eqref{nu}.}
\\  \hline 
 3&    0.158E-02 &  6.4&    0.158E+00 &  4.0&    0.530E-01 &  2.3 \\
 4&    0.152E-03 &  3.4&    0.272E-01 &  2.5&    0.124E-01 &  2.1 \\
 5&    0.187E-04 &  3.0&    0.608E-02 &  2.2&    0.305E-02 &  2.0 \\
  \hline  
\end{tabular} \end{center}  \end{table}
 
  \begin{table}[H]
  \caption{By the $P_2$ element for \eqref{s-1}  on Figure \ref{f21} grids.} \label{t3}
\begin{center}  
   \begin{tabular}{c|rr|rr|rr}  
 \hline 
$G_i$ &\multicolumn{2}{c|}{$ \|Q_h \b u - \b  u_h \| $   $O(h^r)$ }&
 \multicolumn{2}{c|}{$\| \nabla_w( Q_h \b u- \b u_h )\| $ $O(h^r)$ }
   & \multicolumn{2}{c}{$\|\nabla_w(Q_h p-p_h)\|$ $O(h^r)$} \\ \hline 
&\multicolumn{6}{c}{ By the $P_3$-$P_3$/$P_{4}^2$ weak Galerkin finite element, $\nu=0.25$ in \eqref{nu}.}
\\  \hline 
 2&    0.353E-02 &  5.2&    0.299E+00 &  4.4&    0.437E-01 &  4.1 \\
 3&    0.121E-03 &  4.9&    0.197E-01 &  3.9&    0.357E-02 &  3.6 \\
 4&    0.450E-05 &  4.7&    0.138E-02 &  3.8&    0.533E-03 &  2.7 \\
 \hline  
&\multicolumn{6}{c}{ By the $P_3$-$P_3$/$P_{4}^2$ weak Galerkin finite element, $\nu=0.499$ in \eqref{nu}.}
\\  \hline 
 2&    0.459E-01 &  5.9&    0.711E+00 &  5.7&    0.436E-01 &  4.1 \\
 3&    0.416E-03 &  6.8&    0.222E-01 &  5.0&    0.358E-02 &  3.6 \\
 4&    0.507E-05 &  6.4&    0.137E-02 &  4.0&    0.533E-03 &  2.7 \\
  \hline  
\end{tabular} \end{center}  \end{table}

We recompute the solution in \eqref{s-1} by the weak Galerkin finite element
   $P_k$-$P_k$/$P_{k+2}^2$ (for $\{u_0, u_b\}$/$\nabla_w$), $k= 1,2,3$, on nonconvex  
    grids shown in Figure \ref{f22}.
The errors and the computed orders of convergence are listed in Tables \ref{t4}--\ref{t6}.
The optimal order of convergence is achieved in every case, independent of $\nu$.

\begin{figure}[H]
 \begin{center}\setlength\unitlength{1.0pt}
\begin{picture}(360,120)(0,0)
  \put(15,108){$G_1$:} \put(125,108){$G_2$:} \put(235,108){$G_3$:} 
  \put(0,-420){\includegraphics[width=380pt]{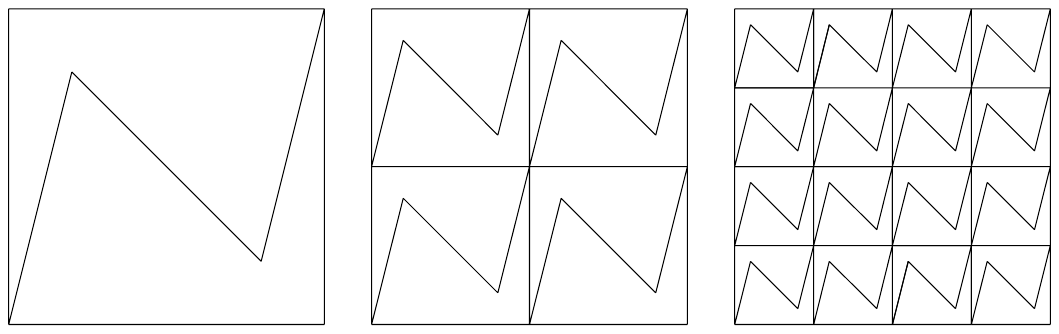}}  
 \end{picture}\end{center}
\caption{The nonconvex polygonal grids used in Tables \ref{t4}--\ref{t6}. }\label{f22}
\end{figure}

  \begin{table}[H]
  \caption{By the $P_1$ element for \eqref{s-1}  on Figure \ref{f22} grids.} \label{t4}
\begin{center}  
   \begin{tabular}{c|rr|rr|rr}  
 \hline 
$G_i$ &\multicolumn{2}{c|}{$ \|Q_h \b u - \b  u_h \| $   $O(h^r)$ }&
 \multicolumn{2}{c|}{$\| \nabla_w( Q_h \b u- \b u_h )\| $ $O(h^r)$ }
   & \multicolumn{2}{c}{$\|\nabla_w(Q_h p-p_h)\|$ $O(h^r)$} \\ \hline 
&\multicolumn{6}{c}{ By the $P_1$-$P_1$/$P_{3}^2$ weak Galerkin finite element, $\nu=0.25$ in \eqref{nu}.}
\\  \hline 
 4&    0.634E-02 &  1.9&    0.387E+00 &  1.0&    0.734E-01 &  1.0 \\
 5&    0.161E-02 &  2.0&    0.194E+00 &  1.0&    0.365E-01 &  1.0 \\
 6&    0.402E-03 &  2.0&    0.974E-01 &  1.0&    0.182E-01 &  1.0 \\
 \hline  
&\multicolumn{6}{c}{ By the $P_1$-$P_1$/$P_{3}^2$ weak Galerkin finite element, $\nu=0.499$ in \eqref{nu}.}
\\  \hline 
 4&    0.319E-02 &  3.5&    0.402E+00 &  2.1&    0.729E-01 &  1.0 \\
 5&    0.637E-03 &  2.3&    0.176E+00 &  1.2&    0.364E-01 &  1.0 \\
 6&    0.153E-03 &  2.1&    0.860E-01 &  1.0&    0.182E-01 &  1.0 \\
  \hline  
\end{tabular} \end{center}  \end{table}

  \begin{table}[H]
  \caption{By the $P_2$ element for \eqref{s-1}  on Figure \ref{f22} grids.} \label{t5}
\begin{center}  
   \begin{tabular}{c|rr|rr|rr}  
 \hline 
$G_i$ &\multicolumn{2}{c|}{$ \|Q_h \b u - \b  u_h \| $   $O(h^r)$ }&
 \multicolumn{2}{c|}{$\| \nabla_w( Q_h \b u- \b u_h )\| $ $O(h^r)$ }
   & \multicolumn{2}{c}{$\|\nabla_w(Q_h p-p_h)\|$ $O(h^r)$} \\ \hline 
&\multicolumn{6}{c}{ By the $P_2$-$P_2$/$P_{4}^2$ weak Galerkin finite element, $\nu=0.25$ in \eqref{nu}.}
\\  \hline 
 4&    0.150E-03 &  3.4&    0.318E-01 &  2.5&    0.358E-02 &  2.1 \\
 5&    0.170E-04 &  3.1&    0.696E-02 &  2.2&    0.886E-03 &  2.0 \\
 6&    0.212E-05 &  3.0&    0.167E-02 &  2.1&    0.222E-03 &  2.0 \\
 \hline  
&\multicolumn{6}{c}{ By the $P_2$-$P_2$/$P_{4}^2$ weak Galerkin finite element, $\nu=0.499$ in \eqref{nu}.}
\\  \hline 
 4&    0.169E-03 &  4.3&    0.318E-01 &  2.6&    0.358E-02 &  2.1 \\
 5&    0.181E-04 &  3.2&    0.675E-02 &  2.2&    0.886E-03 &  2.0 \\
 6&    0.216E-05 &  3.1&    0.160E-02 &  2.1&    0.221E-03 &  2.0 \\
  \hline  
\end{tabular} \end{center}  \end{table}
 
  \begin{table}[H]
  \caption{By the $P_2$ element for \eqref{s-1}  on Figure \ref{f22} grids.} \label{t6}
\begin{center}  
   \begin{tabular}{c|rr|rr|rr}  
 \hline 
$G_i$ &\multicolumn{2}{c|}{$ \|Q_h \b u - \b  u_h \| $   $O(h^r)$ }&
 \multicolumn{2}{c|}{$\| \nabla_w( Q_h \b u- \b u_h )\| $ $O(h^r)$ }
   & \multicolumn{2}{c}{$\|\nabla_w(Q_h p-p_h)\|$ $O(h^r)$} \\ \hline 
&\multicolumn{6}{c}{ By the $P_3$-$P_3$/$P_{5}^2$ weak Galerkin finite element, $\nu=0.25$ in \eqref{nu}.}
\\  \hline 
 2&    0.655E-02 &  4.9&    0.809E+00 &  3.8&    0.148E-01 &  4.2 \\
 3&    0.219E-03 &  4.9&    0.531E-01 &  3.9&    0.122E-02 &  3.6 \\
 4&    0.748E-05 &  4.9&    0.346E-02 &  3.9&    0.127E-03 &  3.3 \\
 \hline  
&\multicolumn{6}{c}{ By the $P_3$-$P_3$/$P_{5}^2$ weak Galerkin finite element, $\nu=0.499$ in \eqref{nu}.}
\\  \hline 
 2&    0.243E+00 &  4.0&    0.349E+01 &  4.4&    0.383E-01 &  4.2 \\
 3&    0.182E-02 &  7.1&    0.135E+00 &  4.7&    0.315E-02 &  3.6 \\
 4&    0.207E-04 &  6.5&    0.875E-02 &  4.0&    0.360E-03 &  3.1 \\
  \hline  
\end{tabular} \end{center}  \end{table}
 
 In the second example, we find the steady-state of
    the Biot's model \eqref{model1}--\eqref{model2} on the unit square domain 
  $\Omega=(0,1)\times(0,1)$, where $\b f=\b 0$, $g=0$ and 
\an{\label{K} K=\begin{cases} 1 & \t{in } (0,\frac 14)\times (0,1), \\
       K_0 & \t{in } [\frac 14,\frac 34]\times (0,1), \\
         1 & \t{in } ( \frac 34, 1)\times (0,1).\end{cases} }
The boundary conditions are 
\a{ & \left\{ \ad {
       \b u&=\p{-\sin (\pi y)\\0} \ \t{on} \ \{1\} \times (0,1),  \\
       \b u&=\b 0 \ \t{on} \ (0,1)\times\{0,1\},  \\
       \partial_{\b n} \b u&=0 \ \t{on} \ \{0\} \times (0,1), } \right. \\
    & \left\{ \ad {
       p&=0 \ \t{on} \ \{1\} \times (0,1),  \\ 
       \partial_{\b n} p&=0 \ \t{on} \ (0,1)\times\{0,1\} \ \t{and} 
                  \ \{0\} \times (0,1). } \right. }
The initial conditions are  
\an{\label{u0} \ad{ \b u&= \p{-\sin (\pi y)\\0} \ \t{and} \ p=0.  } }

We compute the steady state solution for \eqref{u0}
     at $t=1$ by the $P_2$-$P_2$/$P_{4}^2$ weak Galerkin finite element 
  on the fourth nonconvex polygonal grid $G_4$, shown in Figure \ref{f22}.
For $K_0=1$ in \eqref{K}, the computed solution is plotted in Figure \ref{f-s1}.
Near the out-flow boundary, due to $\partial_{\b n} p=0$ at the boundary
     and $\nabla p\approx 0$ nearby, the flow is nearly divergence-free and
   and $(\b u_h)_2$ is quite big.
A similar phenomenon appears at the inflow boundary.

\begin{figure}[H]
 \begin{center}\setlength\unitlength{1.0pt}
\begin{picture}(400,390)(0,0)
  \put(0,-130){\includegraphics[width=400pt]{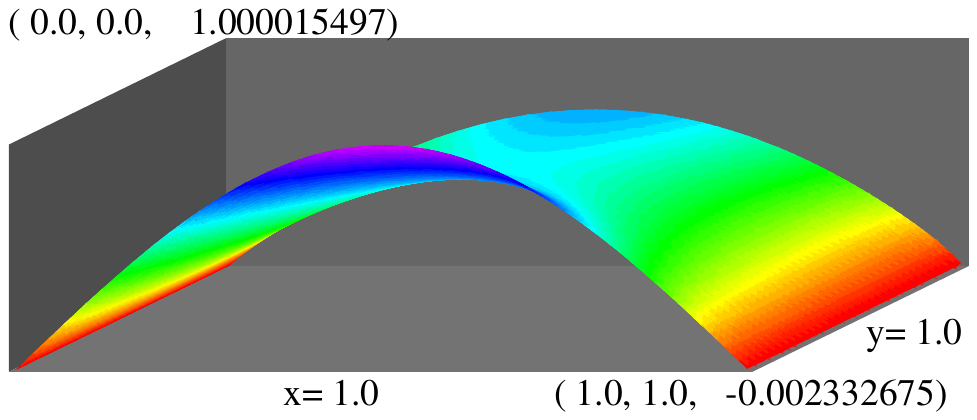}}  
  \put(0,-265){\includegraphics[width=400pt]{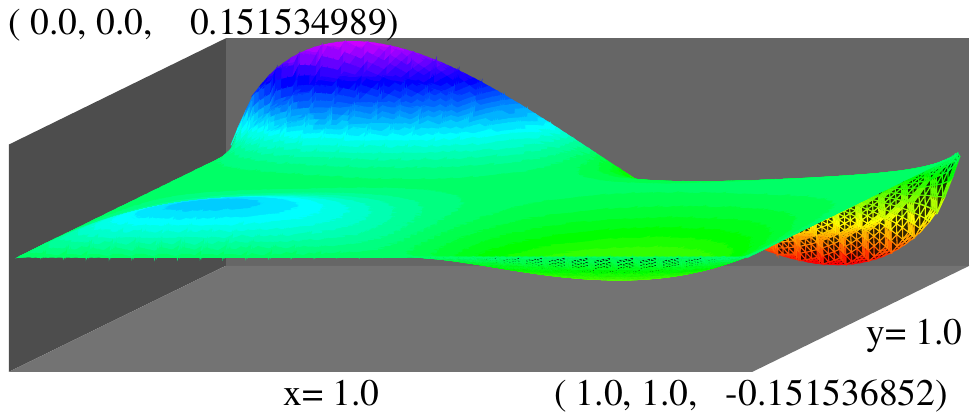}}  
  \put(0,-400){\includegraphics[width=400pt]{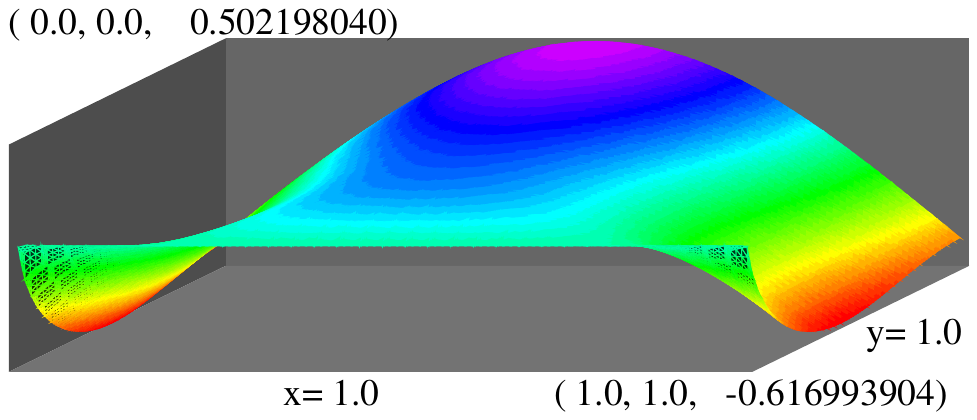}}  
 \end{picture}\end{center}
\caption{The steady state solution for \eqref{u0}, $(\b u_h)_1$(top), $(\b u_h)_2$ 
    and $p_h$, when $K_0=1$ in \eqref{K}.  }\label{f-s1}
\end{figure}
 
 We compute the steady state solution again for \eqref{u0}, with $K_0=10^{-6}$ in \eqref{K} this time,
     at $t=1$ by the $P_2$-$P_2$/$P_{4}^2$ weak Galerkin finite element 
  on the fourth nonconvex polygonal grid $G_4$, shown in Figure \ref{f22}.
From the computed solution plotted in Figure \ref{f-s2}, we have a sharp internal layer for the
  pressure $p$ at the inflow interface $x=\frac 34$.
Due to the discontinuous approximation, we do not have any typical oscillation and interface-smear there.

\begin{figure}[H]
 \begin{center}\setlength\unitlength{1.0pt}
\begin{picture}(400,390)(0,0)
  \put(0,-130){\includegraphics[width=400pt]{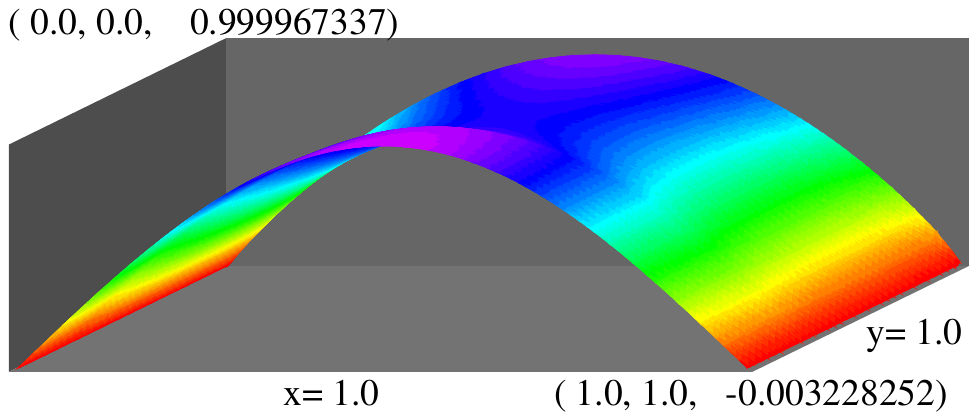}}  
  \put(0,-265){\includegraphics[width=400pt]{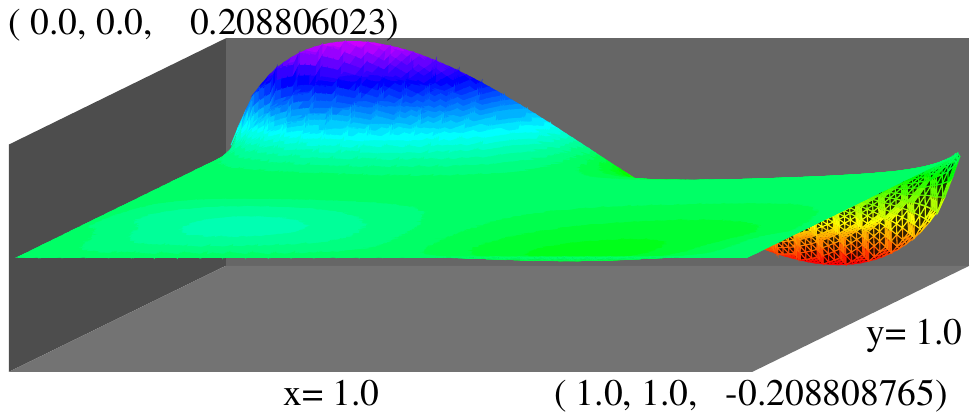}}  
  \put(0,-400){\includegraphics[width=400pt]{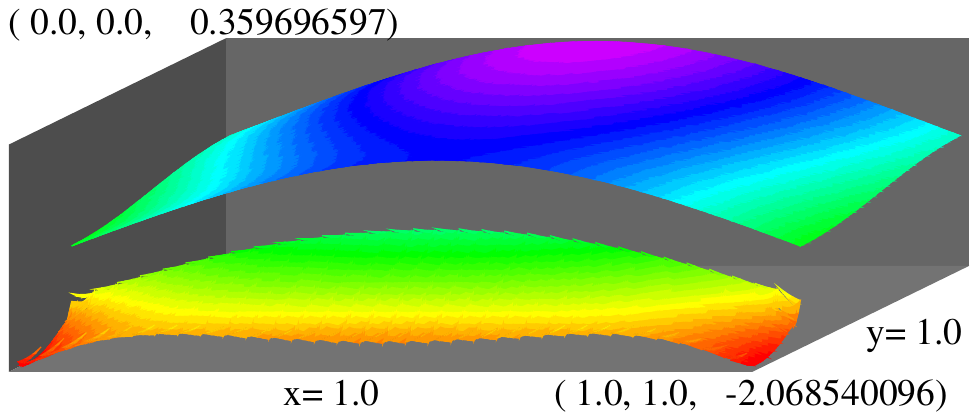}}  
 \end{picture}\end{center}
\caption{The steady state solution for \eqref{u0}, $(\b u_h)_1$(top), $(\b u_h)_2$ 
    and $p_h$, when $K_0=10^{-6}$ in \eqref{K}.  }\label{f-s2}
\end{figure}

\end{document}